\documentclass[11pt]{amsart}
\usepackage[utf8]{inputenc}

\usepackage[english]{babel}
\usepackage[T1]{fontenc}
\usepackage{amsmath, amssymb}

\usepackage[a4paper,top=3cm,bottom=2cm,left=3cm,right=3cm,marginparwidth=1.75cm]{geometry}

\usepackage{amsthm}
\usepackage{amsfonts}
\usepackage{graphicx}
\usepackage[colorinlistoftodos]{todonotes}
\usepackage[colorlinks=true, allcolors=blue]{hyperref}
\usepackage{url}
\newtheorem{theorem}{Theorem}
[section]
\newtheorem{conjecture}[theorem]{Conjecture}
\newtheorem{corollary}[theorem]{Corollary}
\newtheorem{definition}[theorem]{Definition}
\newtheorem{lemma}[theorem]{Lemma}
\newtheorem{proposition}[theorem]{Proposition}

\theoremstyle{remark}
\newtheorem{remark}[theorem]{Remark}
\newtheorem{example}[theorem]{Example}

\def\CC{\mathbb{C}}

\def\NN{\mathbb{N}}
\def\QQ{\mathbb{Q}}

\def\ZZ{\mathbb{Z}}

\title{On periodicity of $p$-adic Browkin continued fractions}

\author[L. Capuano]{Laura Capuano} 
\address{DISMA ``Luigi Lagrange'', Politecnico di Torino, Corso Duca degli Abruzzi 24, 10129
Torino, Italy}
\email{laura.capuano@polito.it}

\author[N. Murru]{Nadir Murru} 
\address{Dipartimento di Matematica, Università di Trento, Via Sommarive 14, 38123, Povo (TN), Italy}
\email{nadir.murru@gmail.com}

\author[L. Terracini]{Lea Terracini} 
\address{Dipartimento di Matematica ``G. Peano'', Università di Torino, Via Carlo Alberto 10, 10123, Torino, Italy}
\email{lea.terracini@unito.it}

\date{}
\subjclass[2010]{11J70, 11D88, 11Y65.} 
\keywords{$p$-adic continued fractions, periodicity, quadratic irrationals.}

\begin{document}

\maketitle

\begin{abstract}
The classical theory of continued fractions has been widely studied for centuries for its important properties of good approximation, and more recently it has been generalized to $p$-adic numbers where it presents many differences with respect to the real case. In this paper we investigate periodicity for the $p$-adic continued fractions introduced by Browkin. We give some necessary and sufficient conditions for periodicity in general, although a full characterization of $p$-adic numbers having purely periodic Browkin continued fraction expansion is still missing. In the second part of the paper, we describe a general procedure to construct square roots of integers having periodic Browkin $p$-adic continued fraction expansion of prescribed even period length. As a consequence, we prove that, for every $n \ge 1$, there exist infinitely many  $\sqrt{m}\in \QQ_p$ with periodic Browkin expansion of period $2^n$, extending a previous result of Bedocchi obtained for $n=1$.
\end{abstract}

\section{Introduction}

Continued fractions provide a representation for any real number as an integer sequence by means of the Euclidean algorithm. They have been widely studied for centuries due to their important properties and applications. In particular, they give the best rational approximations of real numbers and periodic continued fractions characterize quadratic irrationals. For these reasons, it may be interesting to define continued fractions and exploit their usefulness in other fields. In the case of $p$-adic continued fractions, there is no
straightforward generalization of the real algorithm since there is no canonical $p$-adic analogue of the integral part, and several authors as Mahler \cite{Mahler}, Schneider \cite{Sch}, Ruban \cite{Rub} and Browkin \cite{Browkin78}, \cite{Browkin00} proposed different algorithms in the attempt of recovering the same nice properties which hold in the real case. In particular, Browkin definition differs from the other ones because the partial quotients of the expansion are not always positive; in this way, they provide finite expansions for every rational number unlike Ruban and Schneider ones, where rational numbers can have either finite or periodic $p$--adic continued fraction expansion \cite{Lao, Poo, Wang}. Very recently Browkin algorithm has been also extended to multidimensional continued fractions \cite{MT2019, MT2020}. \\

Periodicity of $p$--adic continued fractions is an intriguing feature as well as in the real case. It is easy to show that in general, to be quadratic over $\QQ$ is a necessary condition for periodicity, but the problem of deciding whether this condition is also sufficient is still unknown in full generality. In the case of Schneider continued fractions, it has been proved that the expansion of a $p$-adic unit is periodic if and only if $\alpha_k \cdot \alpha^c_k < 0$, for every complete quotient $\alpha_k$, where $\alpha^c_k$ is its algebraic conjugate, see \cite{Poo, Til, Weg}. A similar property holds also for Ruban continued fractions; namely, the expansion of a quadratic irrational $\alpha$ is periodic if and only if $\QQ(\alpha)$ can be embedded into the reals and $\alpha_k \cdot \alpha_k^c<0$ for every $k$ sufficiently large. 
Indeed, in \cite{Ooto} Ooto showed that $\sqrt m$ with $m<0$ cannot have periodic Ruban continued fraction expansion, and not so much later Capuano, Veneziano and Zannier \cite{CVZ2019} gave an effective criterion for the periodicity of quadratic irrationals. This criterion heavily depends on the property that, for Ruban continued fractions, the partial quotients are always positive. 
In the case of Browkin $p$-adic continued fractions the question appears to be more delicate, and up to our knowledge it is not known whether an analogous of Lagrange's theorem holds. In this context, Bedocchi \cite{Bedocchi1988} proved an analogue of Galois theorem, i.e., if $\alpha \in \mathbb Q_p$ has a periodic expansion, then it is purely periodic if and only if $\lvert \alpha \rvert_p > 1$ and $\lvert \alpha^c \rvert_p < 1$, where $\lvert \cdot \rvert_p$ is the $p$--adic norm. In the first part of the paper, we study periodicity of Browkin continued fraction expansions; more specifically, we first give some properties of $p$-adic numbers with purely periodic Browkin continued fraction expansions, using some simpler arguments with respect to the original ones used by Bedocchi, and we extend some results that hold in the Ruban case, giving some sufficient conditions for periodicity. \\

Another interesting problem in this setting is the study of the possible lengths appearing as periods of square roots of integers in $\QQ_p$. If $m \in \mathbb Z$ and $\sqrt{m} \in \mathbb Q_p$ has periodic Browkin continued fraction expansion, then the length of the pre-period is at most $2$, for $p$ odd. Looking at the lengths of the periods, in \cite{Bedocchi1989}, Bedocchi proved that $\sqrt{m}$ has never period $1$ and found infinitely many square roots of integers in $\QQ_p$ having periodic expansion with period of length $2$. On the other hand, in \cite{Bedocchi1990} he proved that, given an odd integer $d>0$, there are at most finitely many $\sqrt{m}\in \QQ_p$ with $m\in \ZZ$ such that the $p$-adic continued fraction expansion is periodic with period of length $d$. 

In the second part of the paper, we continue this investigation focusing on square roots of integers with even period. In particular, we give a general construction that, starting from a suitable finite continued fraction $[a_0, \ldots, a_{t-1}]$ (that we call \textit{nice}) of length $t$, provides infinitely many $\alpha=p^k \sqrt m$ with $m\in \ZZ$ and $k\ge 1$ with Browkin continued fraction expansion $[0, a_0, \overline{a_1, \ldots, a_{t-1},a_t, a_{t-1}, \ldots, a_1, 2a_0}]$, i.e. of period length $2t$. Using this method we are able to construct, for every $n >0$, infinitely many square roots of integers in $\QQ_p$ having periodic Browkin continued fraction expansion with period length $2^n$. This in particular extends previous results of Bedocchi \cite{Bedocchi1989} on square roots of integers with periodic expansion of period $2$. We finally conjecture that the same should be true for every even length.


\subsection*{Acknowledgements}
The three authors are members of the INdAM group GNSAGA. The first author is also member of DISMA, Politecnico di Torino, Dipartimento di Eccellenza MIUR 2018-2022.

\section{Preliminaries and notation} \label{sec:pre}
In this section we recall the definition of Browkin algorithm that we are going to use in the paper and we provide some preliminary and known results about these $p$--adic continued fractions.

In what follows we will fix a prime $p>2$ and we will denote by $\lvert \cdot \rvert_p$ and $\lvert \cdot \rvert_\infty$ the $p$--adic and the Euclidean norm respectively and by $v_p$ the usual $p$-adic valuation. 

To generalise the usual definition of continued fraction expansion for real numbers, we need first to define a good analogue of the integral part. For $p$-adic numbers there is no canonical way to define it, since, given $\alpha \in \QQ_p$ there are infinitely many $a\in \ZZ$ such that $0\le |\alpha-a|_p < 1$, and there is no canonical choice so that the analogues of theorems about real continued fractions hold. Many
authors (see \cite{Sch, Rub, Browkin78}) gave different definitions of $p$-adic continued fractions with the aim of recovering in this setting the same good properties holding for real numbers. In this paper, we will focus on the definition given by Browkin \cite{Browkin00, Browkin78}, that has the good property of recovering finiteness for $p$-adic continued fraction expansion of rational numbers. We will usually refer to this definition using the abbreviation $BCF$.

Let us define the set $\mathcal{Y}:=\ZZ\left [\frac 1 p\right ]\cap \left (-\frac p2,\frac p 2\right)$. Since $\mathcal{Y}\cap (\mathcal{Y}+p\ZZ)=\emptyset$, we have that $\mathcal{Y}$ is a discrete subset of $\QQ_p$, and given $\alpha \in \QQ_p$ there exists a unique $s(\alpha) \in \mathcal{Y}$ such that $0 \le |\alpha - s(\alpha)|_p <1$. 
We define the \emph{Browkin $s$-function} $s:\QQ_p\longrightarrow \mathcal{Y}$ as the function that associates to any $\alpha \in \QQ_p$ the corresponding $s(\alpha)$.

Given $\alpha \in \mathbb Q_p$, we can determine its $p$--adic continued fraction expansion using the following algorithm:
\begin{equation} \label{eq:algo}
\left\{ \begin{array}{lll}
\alpha_0&=&\alpha, \\
a_n &=& s(\alpha_n),\\ 
\alpha_{n+1}&= & \frac 1 {\alpha_{n}-a_n} \ \ \mbox{if } \alpha_n-a_n \neq 0.
\end{array}\right.
\end{equation}
The $a_n$'s and $\alpha_n$'s are called \emph{partial} and \emph{complete quotients}, respectively.
Analogously to the real case, we define the sequences $(A_n)_{n=-1}^\infty$, $(B_n)_{n=-1}^\infty$ by
\begin{equation*} 
A_{n}=a_nA_{n-1}+A_{n-2}, \quad B_{n}=a_nB_{n-1}+B_{n-2}, \quad n = 1, 2, \ldots,
\end{equation*}
with initial conditions
\begin{equation*} 
A_{-1}=1,\quad A_0=a_0, \quad B_{-1}=0, \quad B_0=1.
\end{equation*}
Using matrices, we can write
$$
 \mathcal{A}_n =\begin{pmatrix} a_n &1 \\
   1 &0 \end{pmatrix},\quad\quad\quad
   \mathcal{B}_n =\begin{pmatrix} {A_{n}} &{A_{n-1}} \\
  {B_{n}} &{B}_{n-1} \end{pmatrix}, 
$$
and we have
\begin{equation}\mathcal{B}_n=\mathcal{B}_{n-1}\mathcal{A}_n
 =\mathcal{A}_0\mathcal{A}_1\ldots \mathcal{A}_n.
\label{eq:matriciale}
\end{equation}
This easily implies by induction the useful relation
\begin{equation} \label{eq:ABrel}
A_n B_{n-1}- B_n A_{n-1}=(-1)^{n+1}  \quad \mbox{for every } n\ge 0.
\end{equation}
Moreover, for every $k \ge 0$ we have 
\begin{equation} \label{eq:alpha}
    \alpha = \cfrac{\alpha_{k}A_{k-1}+A_{k-2}}{\alpha_{k}B_{k-1}+B_{k-2}}.
\end{equation}
For every $n\in \NN$ the quotient $Q_n = \frac{A_n}{B_n}$ gives the
$n$-th convergent of the continued fraction expansion, i.e.
$$ Q_n=[a_0, \ldots, a_n]. $$
Notice that by construction we have
that 
\[ \lvert a_n \rvert_p =\lvert \alpha_n \rvert_p, \quad \lvert a_n \rvert_p > 1, \]
for every $n \geq 1$, and 
\[
\lvert \alpha \rvert_p = \lvert a_0 \rvert_p= \lvert Q_n \rvert_p,
\]
for every $n \geq 0$.
We define the following useful quantities:
\begin{equation*}\label{eq:kappa} k_i:=k_i(\alpha)=-v_p(\alpha_i),\quad K_n:=K_n(\alpha)=\sum_{i=1}^nk_i,\quad K'_n:=K'_n(\alpha)=K_n+k_0;\end{equation*} in this way, for every $n\ge 0$ we have
\[ v_p(A_n) = -K'_n, \quad v_p(B_n) = -K_n,\]
and
\begin{equation} \label{eq:v_p}
v_p(Q_n-\alpha) = 2K_n + k_{n+1}\ge 2n+1.
\end{equation}

By \eqref{eq:v_p} the sequence of convergents $\{Q_n\}_{n\ge 0}$ is a Cauchy sequence with respect to the $p$-adic metric and converges to $\alpha$ (see  \cite{Browkin78}). We denote its limit by 
\[ [a_0, a_1, \ldots]= a_0 + \cfrac{1}{a_1 +\cfrac{1}{\ddots}}, \] 
where, as said before, the limit is computed using the $p$-adic metric. We shall refer to this expansion  as the Browkin continued fraction ($BCF$) for $\alpha$.

For what follows, it will be also useful to consider the sequences $(\tilde A_n)_{n=-1}^\infty$ and $(\tilde B_n)_{n=-1}^\infty$, where we denote by $\tilde x$ the prime-to-$p$ part of an element $x\in \ZZ\left [\frac{1}{p} \right ]$, i.e., the numerator of the number. Since $v_p(a_n)=-k_n$, we have that $\tilde{a}_n=p^{k_n} a_n$; moreover, the following recurrence formulas hold:
\begin{equation} \label{eq:ABtilde}
\tilde A_{n}=\tilde a_n \tilde A_{n-1}  + p^{k_n+k_{n-1}} \tilde A_{n-2}, \quad \tilde B_{n}=\tilde a_n \tilde B_{n-1}  + p^{k_n+k_{n-1}} \tilde B_{n-2}, \quad n = 1, 2, \ldots,
\end{equation} 
with initial conditions
\begin{equation*} 
\tilde A_{-1}=1,\quad \tilde A_0=\tilde a_0, \quad \tilde B_{-1}=0, \quad \tilde B_0=1.
\end{equation*}

We conclude this section with the following proposition, which proves that if two $p$-adic numbers have the same $n$-th convergent then they are sufficiently close with respect to the $p$-adic metric. Namely, we have the following:

\begin{proposition}\label{prop:dueenne}
If $\alpha,\beta\in \QQ_p$ have the same $n$-th convergent in the $p$-adic expansion, i.e. $Q_n^\alpha=Q_n^\beta$, then $|\alpha-\beta|_p<\frac 1 {p^{2n}}$.
\end{proposition}
\begin{proof}
 First notice that the hypothesis $Q_n^\alpha=Q_n^\beta$ is equivalent to say that the first $n+1$ partial quotients of $\alpha$ and $\beta$ are the same. We argue by induction on $n$. 
 
 The claim is certainly true for $n=0$, so assume that $n\geq 1$ and $Q_n^\alpha=Q_n^\beta$. This implies that $a_0=s(\alpha)=s(\beta)$ and $Q_{n-1}^{\alpha_1}=Q_{n-1}^{\beta_1}$, hence $|\alpha_1-\beta_1|_p<\frac 1 {p^{2(n-1)}}$. Since $|\alpha_1\beta_1|_p\geq p^2$ we have
 $$|\alpha-\beta|_p=\left |a_0+\frac 1{\alpha_1}-a_0 - \frac 1 {\beta_1}\right |_p=\frac  {|\alpha_1-\beta_1|_p} {|\alpha_1|_p|\beta_1|_p}< \frac 1 {p^{2(n-1)+2}}= \frac 1 {p^{2n}},$$
 proving the claim.
\end{proof}

\section{Regular quadratic irrationalities} \label{sec:reg}

In this section we focus on quadratic irrational numbers, providing some results about the periodicity of the Browkin algorithm described in Section \ref{sec:pre}.

This problem was one of the main questions raised by Browkin \cite{Browkin78}, who was interested in finding a suitable algorithm for $p$-adic continued fraction satisfying good properties of finiteness and periodicity. In \cite{Browkin78}, Browkin proved that its algorithm  satisfies the good property that $\alpha \in \QQ_p$ has finite $BCF$ if and only if it is rational. On the other hand, he provided some examples of quadratic irrationals which seem not to have periodic $BCF$ expansion, but the problem of deciding whether a quadratic irrational has periodic $BCF$ expansion is still open. This is instead known for other $p$-adic continued fractions; for example, in the case of Ruban continued fraction expansion \cite{Rub}, which is defined using the same algorithm \eqref{eq:algo} but another $s$-function, namely taking the set $\ZZ\left[ \frac{1}{p} \right]$ in place of $\mathcal{Y}$, an effective criterion to decide if a quadratic irrational has periodic continued fraction expansion was given by Capuano, Veneziano and Zannier in \cite{CVZ2019}. Unfortunately, the same criterion does not apply in our case since it strictly depends on the fact that in Ruban expansion $a_n>0$ for all $n\ge 0$. The behaviour of the two expansions is really different; for more, see \cite{CVZ2019}. 
\medskip

It is easy to show that, if $\alpha \in \QQ_p$ has periodic $BCF$ expansion, then it is quadratic irrational, and this depends only on the recurrence formulas satisfied by the convergents. First of all notice that, if $\alpha$ is quadratic over $\QQ$, then all the complete quotients $\alpha_k$ will be quadratic irrationals and will lie in $\QQ(\alpha)$. If the expansion of $\alpha$ is periodic, then one of its complete quotients will be purely periodic, hence without loss of generality we can assume $\alpha$ itself to be purely periodic. Then, there exists $k > 0$ such that $\alpha=\alpha_{k}$. By \eqref{eq:alpha} we have 
\begin{equation*}
    B_{k-1}\alpha^2-(A_{k-1}-B_{k-2})\alpha-A_{k-2}=0.
\end{equation*}
This implies that $[\QQ(\alpha):\QQ]\le 2$; but we know that if $\alpha \in \QQ$, then the BFC expansion of $\alpha$ is finite, which gives that $\alpha$ is quadratic irrational.

In the following, $\alpha \in \QQ_p$ will be quadratic over $\QQ$ and, for any element $x$ in $\QQ[\alpha]$, we will denote by $x^c$ the conjugate of $x$.

\begin{proposition}\label{prop:svilconj}
Assume that $v_p(\alpha)<0$ and $v_p(\alpha^c)>0$. Then, for every $n\geq 0$,
\begin{itemize} 
\item[a)] $v_p(\alpha_n)<0$ and $v_p(\alpha^c_n)>0$;
    \item[b) ]
$-\frac 1 {\alpha_{n+1}^c}=\left [a_n,\ldots, a_0, -\frac 1 {\alpha_{0}^c}\right ]$.
\end{itemize}
\end{proposition}
\begin{proof} $a)$ By construction $v_p(\alpha_{n})<0$ for every $n\geq 1$. For $\alpha_n^c$ we have inductively $$v_p(\alpha_{n}^c)=-v_p(\alpha_{n-1}^c-a_{n-1})=-v_p(a_{n-1})=-v_p(\alpha_{n-1})>0.$$
$b)$ Let us put $\beta_n=-\frac 1 {\alpha_{n}^c}$. By conjugating the expression
\[ \alpha_n = a_n +\cfrac{1}{\alpha_{n+1}}, \]
we get $-\frac 1{\alpha_{n+1}^c}=a_n-\alpha_n^c$, hence $\beta_{n+1}=a_n+\frac 1{\beta_n} $ for $n\geq 0$. Moreover, by $a$),  $v_p(\alpha_n^c)>0$ for $n\geq 0$, so that $s(\beta_{n+1})=a_n$ proving the claim.
\end{proof}

We will call $\alpha$ \emph{regular} if $v_p(\alpha)<0$ and $v_p(\alpha^c)>0$. By Proposition \ref{prop:svilconj}, if $\alpha$ is regular then $\alpha_n$ is regular for every $n \ge 0$. Moreover, in this case
\begin{equation*} v_p(\alpha_{n+1}^c)=-v_p(\alpha_n), \quad v_p(tr(\alpha))=v_p(\alpha),\quad v_p(N(\alpha))>v_p(\alpha),
\end{equation*}
where $tr(\cdot)$ and $N(\cdot)$ denote the trace and the norm of a quadratic irrational.

\begin{proposition}\label{prop:esisteneff}
 If $|\alpha-\alpha^c|_p \geq  \frac 1 {p^{2n}}$, then $\alpha_{n+2}$ is regular.
\end{proposition} 
\begin{proof}
Let us write
$$\alpha=[a_0,a_1,\ldots ],\quad \alpha^c=[b_0,b_1,\ldots]$$
and let $k_0$ be the smallest index $k$ such that $a_k\not=b_k$. By Proposition \ref{prop:dueenne} we have that $k_0\leq n$ and $\alpha_{k_0}\not\equiv\alpha_{k_0}^c\pmod p$. Moreover, it is always true that  $|\alpha_n|_p>1$ for every $n\geq 1$. Therefore, up to relabeling the indices it suffices to show that
\begin{equation*}\label{eq:claim1} 
\hbox{if } \alpha\not\equiv\alpha^c\pmod p, \quad \hbox{ then } |\alpha_2^c|_p<1. 
\end{equation*} 
Since $|\alpha-s(\alpha)|_p<1$, we have $|\alpha^c-s(\alpha)|_p\geq 1$, hence
$$ |\alpha_1^c|_p=\left |\frac 1{\alpha^c-s(\alpha)} \right |_p\leq 1.$$
Now, if $|\alpha_1^c|_p<1$, then $\alpha_1$  (and thus $\alpha_2$) is regular and we are done. If $|\alpha_1^c|_p=1$ , then $|\alpha_1^c-a_1|_p=|a_1|_p>1$, so that 
$$|\alpha_2^c|_p=\frac 1 {|\alpha_1^c-a_1|_p}=\frac 1 {|a_1|_p}<1,$$
proving the claim.
\end{proof}

\begin{proposition}\label{prop:purperconj} Assume that the $BCF$ for $\alpha$ is purely periodic, of the form
$$\alpha=[\overline{a_0,\ldots, a_{N-1}}].$$ Then $\alpha$ is regular,
$$ -\frac 1 {\alpha^c}=[\overline{a_{N-1},\ldots,a_0}] \quad \mbox{and} \quad \alpha^c=[0, \overline{-a_{N-1},\ldots,-a_0}].$$
\end{proposition}
\begin{proof}
Assume that $\alpha=[\overline{a_0,\ldots, a_{N-1}}]$; then, for every $k\ge 1$, the minimal polynomial of $\alpha$ is
\begin{align*}  & X^2-\frac{A_{kN-1}-B_{kN-2}}{B_{kN-1}} X-\frac{A_{kN-2}}{B_{kN-1}}=\\ 
& X^2-\left (\frac{A_{kN-1}}{B_{kN-1}}-\frac {B_{kN-2}}{B_{kN-1}}\right ) X-\frac{A_{kN-2}}{B_{kN-2}}\frac{B_{kN-2}}{B_{kN-1}}.
 \end{align*}
 Then we see that $$\lim_{k}\frac{B_{kN-2}}{B_{kN-1}}=-\alpha^c, $$
 where we consider the limit with respect to the $p$-adic norm. Therefore, in $\QQ_p$, 
 \begin{align*} -\frac 1 {\alpha^c} & = \lim_{k}\frac{B_{kN-1}}{B_{kN-2}}\\
 & =[\overline{a_{N-1},\ldots, a_0}]\end{align*}
 Since $a_0=a_N$, with $N\geq 1$, $v_p(\alpha)=v_p(a_0)<0$. Analogously, $v_p\left (-\frac 1 {\alpha^c}\right) =v_p(a_{N-1})<0$, so that $\alpha$ is regular.
\end{proof}

\begin{proposition} \label{prop:regular}
Assume that the $BCF$ for $\alpha$ is periodic. Then
\begin{itemize}
    \item[a)] it is purely periodic if and only if $\alpha$ is regular.
    \item[b)] the length of the pre-period of $\alpha$ coincides with the smallest $n$ such that $\alpha_n$ is regular.
    \item[c)] let $n_0$ be the smallest natural number greater or equal to  $\frac {v_p(\alpha-\alpha^c)} 2$; then the length of the pre-period of $\alpha$ is less or equal to $n_0 +1$.
\end{itemize}
 
 \end{proposition}
 \begin{proof} $a$)
  If the $BCF$ for $\alpha$ is purely periodic then, by Proposition \ref{prop:purperconj}, $\alpha$  must be regular. Conversely, assume that $\alpha$ is regular and the $BCF$ for $\alpha$ is
  $$[a_0,\ldots, a_{t-1},\overline{a_t,\ldots,a_{t+N-1}}] \hbox{ with } a_{t-1}\not=a_{t+N-1}.$$
  Then $\alpha_t=[\overline{a_t,\ldots,a_{t+N-1}}]$ 
   is purely periodic and Proposition \ref{prop:purperconj} implies that
   $$-\frac 1 {\alpha_t^c}=[\overline{a_{t+N-1},\ldots,a_{t}}].$$
On the other hand, by Proposition \ref{prop:svilconj} we have that 
   $$-\frac 1 {\alpha_t^c}=\left [a_{t-1},\ldots,a_{0},-\frac 1 {\alpha^c} \right ].$$
   Comparing the two expressions gives that $a_{t-1}=a_{t+N-1}$, contrarily to our assumptions. It follows that the $BCF$ for $\alpha$ has pre-period empty, so it is purely periodic.
   \begin{itemize}
   \item[$b$)] Follows immediately from $a$).
   \item[$c$)] By Proposition \ref{prop:esisteneff}, $\alpha_{{n_0}+2}$ is regular. Then the assertion follows from $b$).
   \end{itemize}
   \end{proof}

\begin{remark}
We point out that the results contained in Propositions \ref{prop:purperconj} and \ref{prop:regular} point a) were also proved by Bedocchi \cite{Bedocchi1988} in a different and more elaborated way.
\end{remark}
\medskip
\noindent We end the section by noticing that Proposition \ref{prop:purperconj} implies the following easy corollary:
\begin{corollary}
Assume that $\alpha$ has purely periodic $BCF$ expansion; then it is palindrome and only if $N(\alpha)=-1$.\\
In particular, if $\alpha$ has a palindrome periodic expansion, then $\alpha$ is a real quadratic irrational.
\end{corollary}

Although we are not able to prove non-periodicity, it seems that the condition $N(\alpha)=-1$ does not guarantee in general that the expansion is periodic. For example, if we take $p=5$ and $\alpha=\frac{8+\delta}{5}$ where $\delta\in \QQ_5$ is the square root of $89$ which is congruent to $3$ modulo $5$, then the $BCF$ of $\alpha$ is
\[
[-9/5, -2/5, -59/25, 2/5, -9/5, 23/25, 3/5, 1/5, 51/25, 8/5, 2/5, -7/5, -12/5, 6/5, ...],
\]
which doesn't show a periodic pattern.

\section{Some general criteria for periodicity} \label{sec:criteria}

In this section we are going to prove some general criterion to detect in principle periodicity of Browkin continued fractions.
\medskip

If $\alpha \in \QQ_p$ is quadratic irrational over $\QQ$, then we can always write it as 
\[\alpha = \cfrac{b_0 + \delta}{p^{k_0}c_0},\]
with $b_0, c_0, k_0 \in \mathbb Z$, $p \nmid c_0$, $\delta \in \mathbb Q_p$, $\delta^2 = \Delta$ non square integer. Eventually replacing $b_0, c_0, \Delta $ by $b_0c_0, c_0^2$ and $c_0 \Delta$ we can always assume that the coefficients satisfy the extra condition $c_0 \mid \Delta - b_0^2$. In this way, it is easy to prove that the sequence of complete quotients $\{\alpha_n\}_n\ge 0$ of the $BCF$ expansion of $\alpha$ is given by
\[
\alpha_n=\frac{b_n+\delta}{p^{k_n}c_n},
\]
with $b_n,c_n \in \ZZ$, $p \nmid c_n$, $\delta^2=\Delta$ and $a_n,b_n,c_n$ and $k_n$ satisfying the recurrence formulas
\begin{equation}\label{eq:recurrenciesbc}
\begin{cases} b_n+b_{n+1} & =a_n p^{k_n}c_n\\
p^{k_n+k_{n+1}}c_n c_{n+1}&=\Delta-b_{n+1}^2
\end{cases}
\end{equation}
These formulas hold also for Ruban continued fractions, where the expansion is computed using the same algorithm but another $s$-function, namely taking the set $\ZZ\left [ \frac{1}{p} \right] \cap [0, p)$ in place of $\mathcal Y$. For more about this, see \cite{CVZ2019}. We will usually refer to Ruban continued fraction expansions using the abbreviation $RCF$. 

In the following proposition, we show a general criterion which gives necessary and sufficient condition for periodicity.

\begin{proposition} \label{prop:periodicity}
Let $\alpha \in \QQ_p$ be quadratic over $\QQ$. Then, the $BCF$ expansion is periodic if and only if there exists a sequence $n_t \rightarrow \infty$ such that  $\{|b_{n_t}|\}$ is constant.
\end{proposition}

\begin{proof}
Assume first that $\alpha$ has a periodic $BCF$ expansion. Then, there exist $M_0>0$ and $h>0$ such that, for every $m\ge M_0$, $\alpha_{m+h}=\alpha_m$, hence for every $t\in \NN$, we have $\alpha_m=\alpha_{m+th}$. This means that if we take $n_t:=m+th$, then $b_{m+th}=b_{m}$ as wanted. \\

Conversely, assume that there exists a sequence $\{|b_{k_n}|\}$ which is constant. We call this constant $b$. From \eqref{eq:recurrenciesbc}, we have that 
\[
p^{k_{t_{n-1}}+k_{t_n}}c_{k_{t_{n-1}}}c_{k_{t_n}}=\Delta-b^2.
\]
Notice that the $c_{t_n}$'s are non-zero integers and the $k_{t_n}$'s are all positive; therefore, we have $|p^{k_{t_n}}c_{t_n}|\le |\Delta-b^2|$ for every $n\ge 0$. This implies that there exists a finite number of possibilities for the $k_{t_n}$'s and $c_{t_n}$'s, so that $\alpha_{t_n}=\frac{b+\delta}{p^{k_{t_n}}c_{t_n}}$ varies among a finite range of possibilities, hence there exists $N,M\ge 0$ such that $\alpha_{t_N}=\alpha_{t_M}$, giving periodicity as wanted.
\end{proof}

This implies the following corollary:
\begin{corollary} \label{cor:norm}
For every $n\ge 0$, we denote by $\xi_n$ the image of the complete quotient $\alpha_n$ in $\CC$. If there exists $t_n \rightarrow \infty$ such that $N(\xi_{t_n})<0$, then $\alpha$ has periodic $BCF$ expansion.
\end{corollary}

\begin{proof}
This is an easy consequence of the previous proposition. Indeed, if there exists a sequence of $t_n \rightarrow \infty$ such that $N(\xi_{t_n})<0$, then
$\frac{b_{t_n}^2-\Delta}{p^{2k_{t_n}}{c_{t_n}}^2}<0$, which implies that $|b_{t_n}|\le \Delta$ for every $n\ge 0$. As the $b_{t_n}$'s are integers, this implies that on the sequence $t_n$ they take a finite number of possibilities; hence, there exists a subsequence in which the $b_j$ are constant. Applying Proposition \ref{prop:periodicity}, this implies that $\alpha$ has periodic $BCF$ expansion as wanted.
\end{proof}

\begin{remark}
We point out that the hypothesis of the corollary can happen only in the case in which $\alpha$ can be embedded in the reals, i.e. if $\Delta >0$. In Ruban's case, this is also a necessary condition,
while this is not the case for BCF expansion as shown in the following example.
\end{remark}

\begin{example} Choose any $a_0,a_1\in\mathcal{Y}$ such that $v_p(a_0),v_p(a_1)<0$ and  $-4<a_0a_1<0$. Let $\alpha=[\overline{a_0, a_1}]$. Then, $\alpha$ is a root of the polynomial $a_1X^2-a_0a_1X-a_0$, whose discriminant is $a_0a_1(a_0a_1+4)<0$.
\end{example}

\begin{remark}
In \cite[Proposition 6.4]{CVZ2019}, the authors prove that $\alpha$ has periodic $RCF$ expansion if and only if the sequence ${|b_n|}$ is bounded from above. Their proof uses only the recurrence formulas \eqref{eq:recurrenciesbc}, so it applies also to our case. If we compare it with our Proposition \ref{prop:periodicity}, this implies the following easy consequence:
\end{remark}

\begin{corollary}
Let $\alpha \in \QQ_p$ be quadratic over $\QQ$. Then, there exists a sequence $n_t \rightarrow \infty$ such that  $\{|b_{n_t}|\}$ is constant if and only if the sequence of $\{|b_n|\}$ is bounded from above.
\end{corollary}

In particular, Corollary \ref{cor:norm} implies that we have to study only the case in which the norms of the complete quotients become positive from a certain point on.\\

We also point out that, if the partial quotients $a_n$ of the $BCF$ satisfies $0< a_n < \frac{p}{2}$, then it coincides with the $BCF$ expansion, so for these classes of $\alpha$ we can apply \cite[Theorem 1.3]{CVZ2019} which gives a necessary and sufficient condition to have periodic continued fraction expansion.
This class of $\alpha$ is not empty, as showed in the following example:

\begin{example}
Take $p=3$ and consider $\delta$ the only square root of $37$ in $\QQ_3$ which is congruent to $1$ mod $3$. Take then $\alpha=\frac{1+\delta}{6}$; then the $BCF$ expansion of $\alpha$ is $\left [ \overline{\frac{1}{3} }\right]$, which coincides with the Ruban one.
\end{example}

In the case in which the norms of the (real) embeddings of the complete quotients $\alpha_n$ of the $BCF$ expansion of $\alpha$ are negative for ``enough'' consecutive quotients, then we can conclude that we have periodicity, and we have an effective bound for the period of the $BCF$ expansion, similar to the one obtained in \cite{CVZ2019}. More specifically, we have the following result:  
\begin{proposition}
Let $\alpha \in \QQ_p$ be a quadratic irrational and for every $n>0$ denote by $\xi_n$ and $\xi_n'$ the two images $\alpha_n$ in $\CC$ and by $t=\lfloor \sqrt \Delta \rfloor$. Assume that $\exists\, n_0>0$ such that $ N(\xi_n)<0$ for every $n\in [n_0, n_0+K]$ with $K:=(2t+1)\Delta+1-\frac{t(t+1)(2t+1)}{3}$; then, the $BCF$ expansion of $\alpha$ is periodic of period of length at most $K$. 
\end{proposition}

\begin{proof}
The proof is similar to Step b) in the proof of \cite[Theorem 6.5]{CVZ2019}. Assume that $N(\xi_n)<0$ from a certain $n_0>0$ on; then $\Delta-b_n^2>0$. From the recurrence formulas \eqref{eq:recurrenciesbc} we have that the $c_i$'s have all the same sign; therefore, for every fixed value $b_n$, the second equation of \eqref{eq:recurrenciesbc} implies that the quantity $p^{k_n}c_n$ can assume at most $\Delta-b_n^2$ different values. On the other hand, $b_n$ can assume at most $2t+1$ different values between $-t$ and $t$. Using this we have that, if $N(\xi_n)<0$ for at most $1+\sum_{i=-t}^t (\Delta-i^2)=: K$ steps, we have a repetition in the sequence of the complete quotients, which implies periodicity of the $BCF$ of $\alpha$ as wanted. Finally, $K$ gives an effective estimate for the length of the period of the expansion of $\alpha$, which completes the proof.
\end{proof}

\begin{remark}
We point out that there exist examples of $\alpha$ with periodic $BCF$ expansions and with $N(\xi_n)$ with oscillating signs and also examples with $N(\xi_n)>0$ for every $n>n_0$ for a certain $n_0$, as shown in the following example.
\end{remark}

\begin{example} Let $p=5$ and $\alpha=\frac {-13+\sqrt{19}}{30}$. Then the $BCF$ for $\alpha$ is purely periodic with period 12:
$$\alpha=[\overline{4/5, -11/5, -3/5, -4/25, 274/125, -4/25, -3/5, -11/5, 4/5, 1/5, 24/25, 1/5} ]. $$
It can be verified that all complete quotients have norm $>0$.
\end{example} 

Similarly, in the case in which the norms of the complete quotients have 'oscillating signs' for enough steps, we have a similar criterion to detect periodicity and an effective bound for the length of the period of the $BCF$ expansion. More precisely, we have the following result:

\begin{proposition}
Let $\alpha \in \QQ_p$ be a quadratic irrational and for every $n>0$ denote as before by $\xi_n$ and $\xi_n'$ the two images $\alpha_n$ in $\CC$ and by $t=\lfloor \sqrt \Delta \rfloor$. Assume that $\exists\, n_0>0$ such that $ N(\xi_n)$ and $N(\xi_{n+1})$ have alternating signs for every $n\in [n_0, n_0+2K]$ with $K:=(2t+1)\Delta+1-\frac{t(t+1)(2t+1)}{3}$; then, the $BCF$ expansion of $\alpha$ is periodic of period of length at most $2K$. 
\end{proposition}

\begin{proof}
Assume without loss of generality that there exists a $n_0>0$ such that $N(\xi_n)<0$ for the odd $n\ge n_0$ and that $N(\xi_n)>0$ for the even $n\ge n_0$. This implies that for $n=2m+1$ we have $|b_n|< \sqrt{\Delta}$. Using again the recurrence formulas \eqref{eq:recurrenciesbc}, we have that the $a_n$ and the two subsequences $c_{2n}$ and $c_{2n+1}$ will have alternating signs; in particular, if we restrict to the sequence $c_{4m+1}$, then it will have constant sign. We can argue similarly to the previous proposition on the subsequence $4m+1$; indeed, for every fixed value $b_{4m+1}$, the second equation of \eqref{eq:recurrenciesbc} implies that the quantity $p^{k_{4m+1}}c_{4m+1}$ can assume at most $\Delta-b_{4m+1}^2$ different values. On the other hand, $b_{2m+1}$ can assume at most $2t+1$ different values between $-t$ and $t$. Using this we have that, if $N(\xi_{4m+1})<0$ and $N(\xi_{4m+2})>0$ for at most $1+\sum_{i=-t}^t (\Delta-i^2)=: K$ consecutive $m$'s, we have a repetition in the sequence of the complete quotients, which implies periodicity of the $BCF$ of $\alpha$ as wanted. Finally, $2K$ gives an effective estimate for the length of the period of the expansion of $\alpha$, which completes the proof.
\end{proof}



\section{Periodicity of square roots} \label{sec:sqrt}

In this section we will assume that $\alpha$ is a quadratic irrational such that $tr(\alpha)=0$. Then  $v_p(\alpha)=v_p(\alpha^c)$, so that $\alpha$ is not regular, but we have the following cases:
\begin{itemize}
\item if $v_p(\alpha)=0$, then $v_p(\alpha_1^c)=0$. In this case $\alpha_1$ is not regular and $\alpha_2$ is regular;
\item if $v_p(\alpha)<0$, then $v_p(\alpha^c-a_0)=v_p(-\alpha-a_0)=v_p(2a_0)<0$; in this case $v_p(\alpha_1^c)>0$ and $\alpha_1$ is regular;
\item if $v_p(\alpha)>0$, then $\alpha_1=\frac 1 \alpha$ has trace 0 and $v(\alpha_1)<0$, hence $\alpha_1$ is not regular and $\alpha_2$ is regular.
\end{itemize}
Thus, if the $BCF$ for $\alpha$ is periodic, then the preperiod has length $1$ when $v_p(\alpha)<0$ and $2$ when $v_p(\alpha)\geq 0$.
Assume $v_p(\alpha)<0$ (so that $\alpha_1$ is regular) and let 
$$\alpha=[a_0,a_1,\ldots].$$
Then
$$0=\alpha+\alpha^c=2a_0+\frac 1 \alpha_1+\frac 1{\alpha_1^c},$$
so that
$$-\frac 1 {\alpha_1^c}=2a_0+\frac 1\alpha_1.$$
Assume furtherly that $|a_0|<\frac p 4$, so that $2a_0\in\mathcal{Y}$; then
$$-\frac 1 {\alpha_1^c} =[2a_0, a_1, a_2,\ldots ].$$
If the $BCF$ for $\alpha$ is periodic with period $d$, by Proposition \ref{prop:purperconj} we have that
$$a_{d}=2a_0,\quad a_{d-1}=a_1,\ldots, a_{d-j}=a_j,\hbox { for } j=1,\ldots, \left\lfloor\frac d 2\right\rfloor.$$
It follows that the $BCF$ expansion for $\alpha$ is
\begin{equation}\label{eq:tr0palindromo} 
[a_0,\overline{a_1,a_2,\ldots, a_2, a_1,2a_0}],
\end{equation} 
where
$a_1,a_2,\ldots, a_2, a_1$ denotes a palindromic sequence (of any length, odd or even).
Then, we can write
$$\alpha=[a_0,a_1,a_2,\ldots , a_2,a_1, a_0+\alpha];$$
by \eqref{eq:alpha} we have that
\[
\alpha=\frac {(a_0+\alpha) A_{d-1}+A_{d-2}}{(a_0+\alpha) B_{d-1}+B_{d-2}},
\]
which gives
\[
B_{d-1}\alpha^2 + (a_0B_{d-1}+B_{d-2}-A_{d-1})\alpha -(a_0A_{d-1}+A_{d-2})=0.
\]
Since by assumption $tr(\alpha)=0$, this implies that
\begin{align} 
&B_{d-1}\alpha^2  -(a_0A_{d-1}+A_{d-2})=0, \label{eq:eq1} \\
& A_{d-1}=a_0B_{d-1}+B_{d-2}. \nonumber
\end{align}
\begin{proposition}\label{prop:dt} \
\begin{itemize}
    \item If $d=2t$ is even, then 
\begin{align*} a_0A_{d-1}+A_{d-2}&= A_{t-1}(A_t+A_{t-2}),\\
    B_{d-1}&=B_{t-1}(B_t+B_{t-2}).
\end{align*}
\item If $d=2t+1$ is odd, then 
\begin{align*} 
a_0A_{d-1}+A_{d-2}&= A_t^2+A_{t-1}^2,\\
    B_{d-1}&= B_t^2+B_{t-1}^2.
    \end{align*}
\end{itemize}
\end{proposition}
\begin{proof}
Let us consider the palindromic $BCF$ continued fraction of length $d+1$ given by $[a_0,a_1,\ldots, a_1,a_0]$. By \eqref{eq:matriciale}, for this sequence we have that
$$\mathcal{B}_d=\begin{pmatrix} a_0A_{d-1}+A_{d-2} & A_{d-1}\\ a_0B_{d-1}+B_{d-2} & B_{d-1}\end{pmatrix}=\left\{\begin{array}{ll} \mathcal{B}_t\mathcal{B}_{t-1}^T &\hbox{ if $d=2t$}\\ \mathcal{B}_t\mathcal{B}_{t}^T &\hbox{ if $d=2t+1$}\end{array}\right . $$
The result follows by equaling the coefficients.
\end{proof}
From now on we will focus our attention on the periodicity properties of square roots of integer numbers lying in $\mathbb{Q}_p$. In the late eighties this problem was approached by Bedocchi, who proved the following facts.
\begin{itemize}
    \item There are no square roots of integer numbers having a $BCF$ expansion of period $1$ \cite[Prop. 1]{Bedocchi1989} and $3$ \cite[Prop. 1]{Bedocchi1990}.
    \item For every fixed odd number $k$, there are only finitely many square roots of integer numbers having a $BCF$ expansion of period $k$ \cite[Prop. 1]{Bedocchi1990}. 
    \item For $p>3$, there are infinitely many square roots of integers numbers having a $BCF$ expansion of period $2$ \cite[Prop. 2]{Bedocchi1989}.
\end{itemize}
It is then natural to ask what happens if we consider $BCF$ expansions of other even period length. Inspired by the proof of \cite[Prop. 2]{Bedocchi1989}, we are going to develop a general technique to produce many examples of square roots of integer numbers having a periodic $BCF$ of even period. In particular, Theorem \ref{teo:dueenne} will tell us that, for every natural number $n\ge 1$, there are infinitely many square roots of integer numbers having a $BCF$ expansion of period $2^n$. 

\medskip
Assume $d=2t$; then, by \eqref{eq:eq1} and Proposition \ref{prop:dt}, we have that 
\begin{equation}\label{eq:AB} B_{t-1}(B_t+B_{t-2})\alpha^2=A_{t-1}(A_t+A_{t-2}).
\end{equation} 
Let us now consider $\beta=p^{k}\sqrt{m}$, where $m\in\mathbb{Z}\setminus p\mathbb{Z}$ is not a perfect square and $k\geq 1$, and let us put $\alpha=\beta_1= \frac 1\beta=\frac {\sqrt{m}}{p^k m}$; then, $\alpha$ is a root of
$$p^{2k}m X^2 -1.$$
Our goal is to characterise the integers $p^{2k}m$ such that the $BCF$ expansion of $\alpha$ is periodic of the form \eqref{eq:tr0palindromo} with an even period $d=2t$, i.e.
\begin{align*} \alpha & =[a_0,\overline{a_1,\ldots, a_{t-1},a_t,a_{t-1},\ldots,a_1,2a_0}]\\
\beta & =[0,a_0,\overline{a_1,\ldots, a_{t-1},a_t,a_{t-1},\ldots,a_1,2a_0}].
\end{align*}

If $\alpha^2=\frac{1}{p^{2k}m}$, then \eqref{eq:eq1} and  \eqref{eq:AB} become
\begin{align}
\label{eq:ABm} & B_{t-1}(B_t+B_{t-2})=p^{2k}m A_{t-1}(A_t+A_{t-2}).\end{align}

In the following we will use the notation $\tilde x$, introduced in Section \ref{sec:pre}, to denote the prime-to-$p$ part of an element  $x\in \mathbb Z[\frac 1 p]$. i.e. $\tilde{x}=xp^{v_p(x)}$. Using that $\tilde{a}_n = p^{k_n}a_n$ and since $k=k_0$, by multiplying  \eqref{eq:ABm} by $p^{2K_{t-1}+k_t}$ we have 
\begin{align} \label{eq:ABmtilde}
&\tilde{B}_{t-1}(\tilde{B}_t+p^{k_t+k_{t-1}}\tilde{B}_{t-2})=m\tilde{A}_{t-1}(\tilde{A}_t+p^{k_t+k_{t-1}}\tilde{A}_{t-2}).
\end{align}

In the next definition we shall introduce a class of finite $BCF$s $[a_0,\ldots, a_{t-1}]$, called \emph{nice}. We will show in Theorem \ref{prop:tuno} that every nice $BCF$ can be completed in infinitely many different ways by adding a term $a_t\in\mathcal{Y}$ such that the $p$-adic limit of the infinite $BCF$ of the form $[a_0,\overline{a_1,a_2,\ldots, a_2, a_1,2a_0}]$ is equal to $\frac 1 {p^{k_0}\sqrt{m}}$ for some $m\in\mathbb{Z}$.
\begin{definition}\label{def:nice1} 
A finite $BCF$ $[a_0,\ldots, a_{t-1}]$  is \emph{nice} if:
\begin{itemize}
    \item[a)] $|a_0|_p>1$ and $|a_0|_\infty < \frac p 4$;
    \item[b)] $\left |\frac{ A_{t-1}}{ A_{t-2}}\right |_\infty > \frac 4 p$;
    \item[c)] there exists an integer $q$ such that $\tilde B_{t-1}\mid q \mid \tilde B_{t-1}^2 $ and the class of $q$ modulo $\tilde A_{t-1}^2$ belongs to the multiplicative subgroup generated by the class of $p$.
\end{itemize}
\end{definition}

\begin{example}
The $BCF$ given by $\left [ \frac{1}{p}, \frac{1}{p} \right]$ is nice. Indeed, condition $a$) is easily verified. To prove $b$), notice that
\[
\frac{A_1}{A_0}=a_1+\frac{1}{a_0}=\frac{1}{p}+p=\frac{p^2+1}{p} >\frac{4}{p}.
\]
Finally, in this case, $\tilde{B}_1=\tilde{a}_1=1$, so $c$) holds automatically.
\end{example}
\medskip
We remark that, when the partial quotients are positive, in some cases condition $b)$ of Definition \ref{def:nice1} is easier to verify, as showed by the following proposition:
\begin{proposition}\label{prop:rubannice} Let $[a_0,\ldots, a_{t-1}]$ be a $BCF$ such that $a_0,\ldots, a_{t-1}>0$ and assume that $a_{t-1} > \frac 4 p$. If conditions $a$) and $c$) of Definition \ref{def:nice1} are fulfilled, then $[a_0,\ldots, a_{t-1}]$ is nice.
\end{proposition}
\begin{proof}
 We have  $$\frac{A_{t-1}}{A_{t-2}}=a_{t-1}+\frac{A_{t-3}}{A_{t-2}}.$$
 Since all partial quotients are positive, then
 $$\left | \frac{A_{t-1}}{A_{t-2}}\right |_\infty\geq |a_{t-1}|_\infty >\frac 4 p$$
 by hypothesis, so that condition $b$) of Definition \ref{def:nice1} is also satisfied.
\end{proof}
\begin{remark} Condition $c$) in Definition \ref{def:nice1} is fulfilled in the following particular cases
 \begin{itemize}
     \item $\tilde B_{t-1}=\pm 1$;
     \item $\pm p$ is a primitive root modulo $\tilde{A}_{t-1}^2$. In this case, it is well known that $\tilde{A}_{t-1}$ must be either $2$ or a power of an odd prime. 
 \end{itemize}
These facts can be exploited in order to produce examples, as we shall see in Section \ref{sect:particularcases}.
\end{remark}

In the following theorem we will give the general construction which will allow us, starting from a suitable nice $BCF$, to give infinitely many examples of periodic $BCF$ expansions of square roots of integers in $\QQ_p$ of even period. 

\begin{theorem}\label{prop:tuno} Let $[a_0,\ldots, a_{t-1}]$ be a nice $BCF$; then, there exist infinitely many $a_t\in \mathcal{Y}$  such that the periodic $BCF$ $[a_0,\overline{a_1,\ldots,a_{t-1}, a_t,a_{t-1},\ldots,a_1, 2a_0}]$ converges to a quadratic irrational number of the form $\frac 1 {p^{k_0}\sqrt{m}}$ for some $m\in\mathbb{Z}$. The 
$a_t$ can be chosen of the form $a_t=2c_t$, with $c_t\in \mathcal{Y}$.\end{theorem}
\begin{proof}
Let $[a_0,\ldots, a_{t-1}]$ be a nice $BCF$ of length $t$. By condition $c)$ of Definition \ref{def:nice1}, there exists $q\in \ZZ$ such that $q \equiv p^{\omega} \mod{\tilde{A}_{t-1}^2}$ for some $\omega\in \ZZ$. Eventually adding suitable multiples of the multiplicative order of $p$, we can take $\omega> k_0+2K_{t-1}$. 
Put
\begin{align*} 
b&=\frac {p^\omega - q}{\tilde A_{t-1}^2};\\
k_t & =\omega - k_0 -2K_{t-1};\\
\tilde c &= \frac{ -p^{k_t+k_{t-1}}\tilde{A}_{t-2}+ (-1)^{t-1}\frac q {\tilde{B}_{t-1}}}{\tilde{A}_{t-1}}.
\end{align*}
By construction $k_t>0$ and, as by assumption $\tilde{B}_{t-1} \mid q$, we have that $-p^{k_t+k_{t-1}} \tilde{A}_{t-2}+(-1)^{t-1} \frac{q}{\tilde{B}_{t-1}}$ is an integer. We show that $\tilde c\in\mathbb{Z}$. Indeed,
\begin{align*}
 \tilde{B}_{t-1}\left ( -p^{k_t+k_{t-1}}\tilde{A}_{t-2}+ (-1)^{t-1} \frac q {\tilde{B}_{t-1}}\right ) &=  -p^{k_t+k_{t-1}}\tilde{A}_{t-2}\tilde{B}_{t-1}+(-1)^{t-1}q,
\end{align*}
and, since by \eqref{eq:ABrel} $\tilde{A}_{t-1}\tilde{B}_{t-2}-\tilde{A}_{t-2}\tilde{B}_{t-1}=(-1)^tp^{k_0+2K_{t-2}+k_{t-1}}$, using the definitions of $b$ and $k_t$ we get
\begin{align*}
\tilde{B}_{t-1}\left ( -p^{k_t+k_{t-1}}\tilde{A}_{t-2}+ (-1)^{t-1} \frac q {\tilde{B}_{t-1}}\right ) &= (-1)^tp^\omega - p^{k_t+k_{t-1}} \tilde{A}_{t-1}\tilde{B}_{t-2}+(-1)^{t-1}q\nonumber\\
    &=(-1)^t b\tilde A_{t-1}^2 - p^{k_t+k_{t-1}} \tilde{A}_{t-1}\tilde{B}_{t-2}\nonumber \\
    &= \tilde{A}_{t-1}( (-1)^t b\tilde A_{t-1} -p^{k_t+k_{t-1}} \tilde{B}_{t-2}).
\end{align*}
Since $(\tilde{A}_{t-1},\tilde{B}_{t-1})=1$, it follows that $\tilde{A}_{t-1}$ divides  $ -p^{k_t+k_{t-1}}\tilde{A}_{t-2}+ (-1)^{t-1} \frac q {\tilde{B}_{t-1}}$, proving that $\tilde{c}\in \ZZ$.
Moreover by the above calculations we obtain
\begin{align}
    \tilde c \tilde A_{t-1}+p^{k_t+k_{t-1}}\tilde A_{t-2} & = (-1)^{t-1}\frac q {\tilde B_{t-1}}\label{eq:tildeAt},\\
    \tilde c \tilde B_{t-1}+p^{k_t+k_{t-1}}\tilde B_{t-2} & =(-1)^tb\tilde A_{t-1}.\label{eq:tildeBt}
\end{align}
We want now to show that, if $\omega$ is sufficiently large, then we have $\frac{2\tilde c}{p^{k_t}}\in \mathcal{Y}$. Indeed, since $q\mid\tilde{B}_{t-1}^2$, we have
$$|\tilde c|_\infty \leq  p^{k_t+k_{t-1}} \left |\frac{\tilde{A}_{t-2}}{\tilde{A}_{t-1}}\right |_\infty + \left |\frac {\tilde{B}_{t-1}} {\tilde{A}_{t-1}}\right |_\infty.$$
Using the recurrence formulas \eqref{eq:ABtilde}, we have
$$\frac{\tilde{A}_{t-1}}{\tilde{A}_{t-2}}= p^{k_{t-1}} \frac{A_{t-1}}{{A}_{t-2}}\quad \mbox{and} \quad \frac{\tilde{A}_{t-1}}{\tilde{B}_{t-1}}= p^{k_0} \frac{{A}_{t-1}}{{B}_{t-1}}, $$
so that 
\begin{align*} |\tilde c|_\infty \leq  {p^{k_t}}\left |\frac{{A}_{t-2}}{{A}_{t-1}}\right |_\infty  +\frac 1 {p^{k_0}} \left |\frac{{B}_{t-1}}{{A}_{t-1}}\right |_\infty & = p^{k_t}\left (\left |\frac{{A}_{t-2}}{{A}_{t-1}}\right |_\infty + \frac 1 {p^{k_0+k_t}} \left |\frac{{B}_{t-1}}{{A}_{t-1}}\right |_\infty \right ) \\ & < p^{k_t}\left (\frac p 4 + \frac 1 {p^{k_0+k_t}} \left |\frac{{B}_{t-1}}{{A}_{t-1}}\right |_\infty \right ),
\end{align*}
where we used condition $b$) of Definition \ref{def:nice1}. Now, if we choose $\omega$ large enough (and hence $k_t$), we have $|\tilde{c}|_{\infty} < \frac{p^{k_t+1}}{4}$, so $\left |\frac{2\tilde{c}}{p^{k_t}} \right |_{\infty}<\frac{p}{2}$ as wanted.
We can now put $a_t=\frac{2\tilde c}{p^{k_t}}$ and consider the finite $BCF$ given by 
$[a_0,a_1,\ldots, a_{t-1},a_t]$.
By \eqref{eq:tildeAt} and \eqref{eq:tildeBt} we have
\begin{align*}\tilde A_{t}+p^{k_t+k_{t-1}}\tilde A_{t-2} &=2(\tilde c \tilde A_{t-1} +p^{k_t+k_{t-1}}\tilde A_{t-2})\\
&= 2(-1)^{t-1}\frac q {\tilde B_{t-1}};\\
\tilde B_{t}+p^{k_t+k_{t-1}}\tilde B_{t-2} &=2(\tilde c \tilde B_{t-1} +p^{k_t+k_{t-1}}\tilde B_{t-2})\\
&=2(-1)^t b\tilde A_{t-1}.
\end{align*} 
Therefore $2\tilde A_{t-1}$ divides $\tilde B_{t}+p^{k_t+k_{t-1}}\tilde B_{t-2}$ and $\tilde A_{t}+p^{k_t+k_{t-1}}\tilde A_{t-2}$ divides $2\tilde B_{t-1}$. Let $q_1=\frac{\tilde B_{t-1}^2}{q}$; then, 
\begin{align*} 2\tilde B_{t-1} (\tilde B_{t}+p^{k_t+k_{t-1}}\tilde B_{t-2}) & =(-1)^{t-1}q_1(\tilde A_t+p^{k_t+k_{t-1}}\tilde A_{t-2})(-1)^t2b\tilde A_{t-1} \\
&= - 2bq_1 \tilde A_{t-1}(\tilde A_t+p^{k_t+k_{t-1}}\tilde A_{t-2}).
\end{align*}
Then equation \eqref{eq:ABmtilde} holds with $m=-bq_1$, proving that, if we take $\alpha$ the limit of the $BCF$ given by
$[a_0,\overline{a_1, \ldots, a_{t-1}, a_t, a_{t-1}, \ldots, a_1, 2a_0 }]$, then it is a root of the polynomial $p^{2k}m X^{2}-1=0$ as wanted.
 \end{proof}
 \subsection{Conjectures about niceness}
 
Using Theorem $\ref{prop:tuno}$, if we start with a nice sequence $[a_0,\ldots, a_{t-1}]$, one can provide a collection of integers $m$, coprime to $p$, such that the $BCF$ expansion of $ {p^{k_0}\sqrt m}$ is periodic of period $2t$. The existence of a nice sequence of length $t$ for every $t\geq 1$ has some experimental support (see Section \ref{sect:particularcases}). We  formulate the following: 

\begin{conjecture}
For every $t\geq 1$ there exists a nice $BCF$ of length $t$, except when $t=1$ and $p=3$.
\end{conjecture}

We also obtain experimental confirmations for the following stronger assertions (with the exception $p=3$, $t=1$): 

\begin{conjecture}
For every $a_0,\ldots, a_{t-2}>0$ in $\mathcal{Y}$ with $|a_0|_\infty < \frac p 4$ there exists $a_{t-1}\in \mathcal{Y}$ such that the $BCF$ $[a_0,\ldots, a_{t-1}]$ is nice.
\end{conjecture}

\begin{conjecture}
For every $t\geq 1$ there exists $a_{t+1}\in \mathcal{Y}$ such that the $BCF$ $\left [\frac 1 p,\ldots,\frac 1 p, a_{t-1}\right ]$ is nice.
\end{conjecture}
By Theorem \ref{prop:tuno}, the truth of any of the above assertions would imply the following result:
\begin{conjecture}\label{prop:conjprin} For every prime $p$ and for every $t\geq 1$ except in the case $p=3$, $t=1$ there are infinitely many integers $p^{2k}m$ with $k \geq 1 $ and $m\in\mathbb{Z}\setminus p\mathbb{Z}$ not a square such that $\alpha=\frac 1 {p^k\sqrt{m}}$ has a $BCF$ of the form $$[a_0,\overline{a_1,\ldots,a_{t-1}, a_t,a_{t-1},\ldots,a_1, 2a_0}]$$
and $a_t=2c$, with $c\in \mathcal{Y}$.\\
In particular  there are infinitely many $b\in\mathbb{Z}$ such that the $BCF$ expansion of $\sqrt{b}$ has period $2t$.
\end{conjecture}

\subsection{Some particular cases}\label{sect:particularcases}

In this section we are going to prove some cases of Conjecture \ref{prop:conjprin} by providing nice $BCF$ expansions of given length for some particular $t$.  

\subsubsection{The case $t=1$}\label{subsubsect:tdue} This case was proven in \cite[Proposition 2]{Bedocchi1989}. Here $\tilde B_{t-1}=\tilde B_0=1$ and $\frac{A_1}{A_0}=a_0$, so that condition $c)$ in Definition \ref{def:nice1} is always satisfied and, if $p\ge 5$, then every $BCF$ $[a_0]$ such that $|a_0|_p>1$ and $\frac 4 p< |a_0|_\infty <\frac p 4$ is nice. By Theorem \ref{prop:tuno} there exist infinitely many $a_1\in \mathcal{Y}$ such that the periodic $BCF$
$[a_0,\overline{a_1,2a_0}]$ represents a quadratic irrational of the form $\frac 1 {{p^k}\sqrt m}$, with $m\in \ZZ$.\\
\begin{example}
Take $p=5$, $a_0=\frac 6 {5}$. The order of $5$ in $\mathbb{Z}_{36}^\times$ is $6$. By setting $\omega=6,12,18\ldots$ in the proof of Theorem \ref{prop:tuno} we obtain the following $BCF$'s:
\begin{itemize}
    \item $[\frac 6 {5},\overline{ -\frac {2604}{ 3125}, \frac{12}{5}}]=\frac 1 {5\sqrt{-434}}$;
    \item $[\frac 6 {5},\overline{ -\frac {40690104}{ 48828125},\frac{12}{5}}]=\frac 1 {10\sqrt{-1695421}}$;
    \item $[\frac 6 {5},\overline{ -\frac {635782877604}{ 762939453125},\frac{12}{5}}]=\frac 1 {5\sqrt{-105963812934}}$.
\end{itemize}
\end{example}

\subsubsection{The case $t=2$}
In this case $\tilde B_{t-1}=\tilde B_1=\tilde{a}_1$. Notice that, if $a_0\in \mathcal{Y}$ is such that $0< a_0< \frac p 4$, then, for every $h\geq 1$, the sequence $\left [a_0, \frac 1 {p^h} \right ]$ is nice. Indeed for this sequence $\tilde B_1=1$, so $c$) is satisfied, and $\frac {A_1}{A_0}=\frac 1 {p^h} +\frac 1 {a_0} > \frac 4 p$.
Then, Theorem \ref{prop:tuno} allows to find infinitely many $a_2\in \mathcal{Y}$ such that the periodic $BCF$
$\left [a_0 ,\overline{\frac 1 {p^h} ,a_2,\frac 1 {p^h}, 2 a_0} \right ]$ represents   a quadratic irrational of the form $\frac 1 {{p^k}\sqrt m}$, with $m\in \ZZ$.
\begin{example}
Take $p=3$, and consider the $BCF$ $\left[ \frac 1 3,\frac 1 3\right]$. Then $\tilde A_1=10$, and the order of $3$ in $\mathbb{Z}_{100}^\times$ is $20$. By setting $\omega=20,40,60 \ldots$ in the proof of Theorem \ref{prop:tuno} we obtain the following $BCF$'s:
\begin{itemize}
    \item $[\frac 1 3,\overline{ \frac 1 3, -\frac {38742049}{ 129140163}, \frac{1}{3}, \frac{2}{3}}]=\frac 1 {66\sqrt{-72041}}$;
    \item $[\frac 1 3, \overline{ \frac 1 3, -\frac {135085171767299209}{ 450283905890997363} \frac{1}{3}, \frac{2}{3}}]=\frac 1 {66\sqrt{-251191435104482}}$;
    \item $[\frac 1 3, \overline{ \frac 1 3, -\frac {471012869724624483492160369}{ 1570042899082081611640534563} \frac{1}{3}, \frac{2}{3}}]=\frac 1 {66\sqrt{-875850377587111642857323}}$;
\end{itemize}
\medskip
Another suggestive way to obtain examples  is to force $\tilde A_{1}$ to be a prime $\ell$ such that $p$ is a primitive root modulo $\ell^2$.
\end{example}
Fix an odd prime $p\geq 5$. Artin conjecture (see \cite{Hooley}) on primitive roots predicts that the density of primes $\ell$ such that $p$ is a primitive root modulo $\ell$ is $\sim C_p \frac {\pi(x)} x$ for $x\to\infty$, where $C_p$ denotes Artin's constant. Moreover, it is well-known that, if $a$ is a primitive root modulo $\ell$, then exactly $\ell-1$ between the $\ell$ liftings $a+t\ell$ with $t=0,\ldots,\ell-1$ are primitive roots modulo $\ell^2$. 
The above considerations justify to conjecture the existence of an integer $k_1>1$ (depending on $p$) such that there is a  prime $\ell$ in the interval
$$p^{k_1+1}+4p^{k_1-1}
< \ell < \frac 3 2 p^{k_1+1}$$
and $p$ is a primitive root modulo $\ell^2$. Put $\tilde a_1=\ell-p^{k_1+1}$ and $a_1=\frac{\tilde a_1}{p^{k_1}}$. Then $a_1\in\mathcal{Y}$ and we can consider the $BCF$ $[\frac 1 p, a_1]$. We have
$$\left |\frac {A_1}{A_0}\right |_\infty = \frac {|\tilde a_1 +p^{k_1+1}|_\infty }{ p^{k_1}}>\frac 4 p.$$
hence $[\frac 1 p, a_1]$ is nice
and Theorem \ref{prop:tuno} gives infinitely many $a_2\in \mathcal{Y}$ such that the periodic $BCF$
$$\left [\frac 1 p,\overline{a_1,a_2,a_1, \frac 2 p}\right ]$$ represents a quadratic irrational of the form $\frac 1 {{p^k}\sqrt m}$, with $m\in \ZZ$.
\begin{example}
Let $p=3$ and $k_1=4$; then, $3$ is a primitive root modulo $353^2$ and $3^5+4\cdot 3^3< 353 < \frac {3^5} 2$. We put $\tilde a_1=353-3^5=110$, and consider the $BCF$ $\left [\frac 1 3, \frac {110}{81}\right ]$. The discrete logarithm of $110$ in base $3$ modulo $353^2$ is $\omega_0=31861$, and the multiplicative order of $3$ modulo $353^2$ is $s=124256$. Following the proof of Theorem \ref{prop:tuno} we construct
\begin{align*}
    b &= \frac{3^{31861}-110}{353^2},\\
    k_2 &= 31861 -9=31852,\\
    \tilde c &= \frac {-3^{31856} -1}{353},\\
    \tilde a_2 &=2\tilde c,\\
    a_2 &= \frac{\tilde a_2} {3^{31852}}.
\end{align*} By setting $\omega=\omega_0+hs$ with $h\in\mathbb{N}$ we obtain infinitely many other examples.
\end{example}
\subsubsection{The case $t=3$}
\begin{proposition}\label{prop:tre} The $BCF$
$\left [\frac{1}{p}, \frac{1-p}{p}, \frac{1+p}{p}\right ]$ is nice.  Therefore there are infinitely many integers $b$ such that the BCF expansion of $\sqrt{b}$ has period $6$.
\end{proposition}
\begin{proof}
 We have $B_2=\frac {1-p^{2}} {p^{2}} +1$ so that $\tilde B_2=1$. Moreover $\frac{A_2}{A_1}=\frac {p^{3}+p^{2}+1}{p(p^{2}-{p}+1)}> \frac 4 p$; indeed the polynomial function
 $$x^3-3x^2+4x-3$$ is increasing and has a positive value in $x=1$.\\
\end{proof}
More generally, it is possible to show that for $p\geq 5$ and every integer $h\geq 1$ the $BCF$
$\left [\frac{1}{p^h}, \frac{1-p^h}{p^h}, \frac{1+p^h}{p^h}\right ]$ is nice. Then, Theorem \ref{prop:tuno} implies that for every $h$ there are infinitely many integers $m$ coprime with $p$ such that the $BCF$ expansion for $p^h\sqrt{m}$ has the form
\[
\left [0, \overline{ \frac{1}{p^h}, \frac{1-p^h}{p^h}, \frac{1+p^h}{p^h}, a_4, \frac{1+p^h}{p^h}, \frac{1-p^h}{p^h}, \frac{2}{p^h} } \right ];
\]
in particular, it has period $6$.

\subsubsection{The case $t=5$}
\begin{proposition}\label{prop:cinque} The $BCF$
$\left [\frac{1}{p}, -\frac{p^3-1}{2p^2},\frac{1}{p}, -\frac 2{p^2}, \frac{1}{p}\right ]$ is nice.  Therefore there are infinitely many integers $b$ such that the $BCF$ expansion of $\sqrt{b}$ has period $10$.
\end{proposition}
\begin{proof}
 We have $B_4=-\frac {1} {p^{6}} $ so that $\tilde B_4=-1$, and $\frac{A_4}{A_3}=\frac {4p^{6}-4p^{3}-2}{p(p^{6}-5{p^3}-2)}> \frac 4 p$.
\end{proof}

\subsubsection{The case $t=2^n$}
This section is devoted to exhibit, for every $n,k\geq 1$, a nice sequence $\beta_n^k$ of length $2^n$ having all partial quotients with denominator $p^k$. Notice that the case $n=0$ has been dealt with in Section \ref{subsubsect:tdue}. \\
In what follows, it will be sometimes useful to write $\tilde B_r(a_0,\ldots, a_n)$ ($n\geq r$), in order to put in evidence the dependence on partial quotients.
\begin{lemma}\label{lem:cala}
Consider a $BCF$ of the form
$$\alpha=\left [\frac {\tilde a_0} {p^k},\frac {1} {p^k},\frac {\tilde a_2} {p^k}, -\frac {1} {p^k}\ldots, \frac {\tilde a_{2j}} {p^k}, (-1)^j \frac {1} {p^k},\ldots \right],$$
 and define
$$\alpha^\bullet=\left [\frac {\tilde a_0} {p^{2k}},- \frac {\tilde a_2} {p^{2k}}, \frac {\tilde a_4} {p^{2k}}\ldots,(-1)^j\frac {\tilde a_{2j}} {p^{2k}},\ldots \right].$$
Put $B^\bullet_j=B_j(\alpha^\bullet)$. Then, for every $i\in \mathbb{N}$,
\begin{itemize}
    \item[a)] $B_{2i}=(-1)^i(B^\bullet_i-B^\bullet_{i-1})$;
    \item[b)] $B_{2i+1}=\frac 1 {p^k} B_i^\bullet$;
    \item[c)] $A_{2i+1}=A_i^\bullet+B_i^\bullet$;
    \item[d)] $A_{2i}= (-1)^ip^k(A_i^\bullet-A_{i-1}^\bullet+B_{i}^\bullet-B_{i-1}^\bullet) $.
\end{itemize}
\end{lemma}
\begin{proof}
 We prove the Lemma by induction on $i$. The claim is certainly true for $i=0$; indeed, 
 $$
 \begin{array}{llll}  A_0= \frac{\tilde a_0} {p^k}, \quad  & A_1= \frac{\tilde a_0}{p^{2k}}+1, \quad &
  B_0= 1,\quad  & B_1= \frac{1}{p^{k}} \\ \ &&&
  \\
 A_{-1}^\bullet =1, & A_0^\bullet =\frac {\tilde a_0} {p^{2k}} ,
 &B_{-1}^\bullet =0,\quad & B_0^\bullet =1.
 \end{array}
 $$
 Assume now by inductive hypothesis that the thesis holds true for $i$.  Then,
  \begin{align*} 
  B_{2i+2} & = \frac {\tilde a_{2i+2}} {p^k} B_{2i+1}+B_{2i}; \quad\hbox{by inductive hypothesis}\\
 & = \frac {\tilde a_{2i+2}} {p^{2k}} B_{i}^\bullet+(-1)^i(B^\bullet_i-B^\bullet_{i-1})\\
 &= (-1)^{i+1}\left ( (-1)^{i+1}\frac  {\tilde a_{2i+2}} {p^{2k}} B_{i}^\bullet+ B_{i-1}^\bullet-B_i^\bullet\right )\\
 &= (-1)^{i+1}(B_{i+1}^\bullet-B_i^\bullet); \\
 B_{2i+3} &= \frac {(-1)^{i+1}}{p^k}B_{2i+2}+B_{2i+1}; \quad\hbox{by inductive hypothesis}\\
 &= \frac 1 {p^k}\left (B_{i+1}^\bullet-B_i^\bullet+B^\bullet_i\right )\\
 &= \frac 1 {p^k}B_{i+1}^\bullet;\\
 A_{2i+3} &= \frac {(-1)^{i+1}}{p^k}A_{2i+2}+A_{2i+1}; \quad\hbox{by  inductive hypothesis}\\
 &= \frac {(-1)^{i+1}}{p^k}((-1)^{i+1}p^k(A_{i+1}^\bullet-A_i^\bullet+B_{i+1}^\bullet-B_i^\bullet)) +A_i^\bullet+B_i^\bullet\\
 &= A_{i+1}^\bullet+B_{i+1}^\bullet; \\
A_{2i+2} & = \frac {\tilde a_{2i+2}} {p^k} A_{2i+1}+A_{2i}; \quad\hbox{by inductive hypothesis}\\
 & = \frac {\tilde a_{2i+2}} {p^{k}}(A_i^\bullet+B_i^\bullet) +(-1)^ip^k(A_i^\bullet-A_{i-1}^\bullet+B_{i}^\bullet-B_{i-1}^\bullet)\\
 &= (-1)^{i+1}p^k\left ( (-1)^{i+1}\frac  {\tilde a_{2i+2}} {p^{2k}} A_{i}^\bullet+ A_{i-1}^\bullet\right )+(-1)^{i+1}p^k\left ((-1)^{i+1}\frac  {\tilde a_{2i+2}} {p^{2k}} B_{i}^\bullet+ B_{i-1}^\bullet\right )+\\
  &\quad + (-1)^ip^k(A_i^\bullet+B_i^\bullet)\\
 &= (-1)^{i+1}p^k(A_{i+1}^\bullet+B_{i+1}^\bullet-A_i^\bullet-B_i^\bullet). 
 \end{align*}
\end{proof}
For $n,k\geq 1$  we can construct inductively a $BCF$
$\beta_n^k$ of length $2^n$ as follows
\begin{itemize} 
\item $\beta_1^k=\left [\frac 1 {p^k}, \frac 1 {p^k}\right ]; $ 
\item if $\beta_n^k=[b_0,\ldots, b_{2^n-1}] $, then
$$\beta_{n+1}^k=\left [b_0,\frac 1 {p^k}, - b_1,-\frac 1 {p^k}, b_2,\frac 1 {p^k}, \ldots, (-1)^ib_i,(-1)^i \frac 1 {p^k},\ldots,  b_{2^n-1}, -\frac 1 {p^k}\right ]$$
\end{itemize}

\begin{proposition}\label{prop:betanice} For every $n,k\geq 1$, $\beta_n^k$ is nice, and $\tilde B_{2^n-1}(\beta_n^k)=1$.
\end{proposition}
\begin{proof}
 Consider the following system of recurrence relations:
\begin{equation}\label{eq:sistema} \left\{ \begin{array}{ll}
x(n,k)&=x(n-1,2k)+y(n-1,2k),\\ y(n,k)& = \frac 1 {p^k}y(n-1,2k),\\
 z(n,k)&=-p^k(x(n-1,2k)-z(n-1,2k)+y(n-1,2k)-w(n-1,2k)),\\
    w(n,k) & = -z(n-1,2k)+w(n-1,2k). \end{array}\right.\end{equation}

By Lemma \ref{lem:cala} the quadruple $(A_{2^{n}-1}(\beta_n^k),B_{2^{n}-1}(\beta_n^k),A_{2^{n}-2}(\beta_n^k),B_{2^{n}-2}(\beta_n^k))$ is a solution of   \eqref{eq:sistema}, subjected to the initial conditions:
\begin{equation}\label{eq:condizioniiniziali} x(1,k)=\frac 1 {p^{2k}}+1, \quad  y(1,k)= \frac 1 {p^k},\quad  z(1,k)=\frac 1 {p^k},\quad  w(1,k)= 1. \end{equation}
In order to determine explicit formulas for $A_{2^{n}-1}(\beta_n^k)$, $B_{2^{n}-1}(\beta_n^k)$,$A_{2^{n}-2}(\beta_n^k)$, $B_{2^{n}-2}(\beta_n^k)$,
we introduce the following polynomials, for $n\geq 1$, 
\begin{align}
   \tilde U(n,X) &= \sum_{j=1}^n X^{2^j}+1; \nonumber \\
   \tilde V(n,X) &= X^{2(2^{n-1}-1)}-\sum_{j=0}^{n-2}X^{2(2^j-1)}; \nonumber \\
   \tilde S(1,X)=1; &\quad   \tilde S(n+1,X)=X^2\tilde S(n, X^2)-1-X^2-2\sum_{j=2}^n X^{2^j}, \label{eq:Stilde}
\end{align}
and define the sequences of rational functions
\begin{align*}
  U(n,X) = \frac{\tilde U(n,X)} {X^{2^n}}, \qquad 
  & V(n,X) = \frac{\tilde V(n,X)} {X^{2^n-2}}, \nonumber \\
   S(n,X) = \frac{\tilde S(n,X)} {X^{2^n-1}}, \qquad  &
   W(n,X) = \frac{1}{X^{2^n-1}}. 
\end{align*}
Then, the quadruple  $(S(n,p^k),U(n,p^k),V(n,p^k),W(n,p^k))$ is a solution of system \eqref{eq:sistema} subjected to she same initial conditions \eqref{eq:condizioniiniziali}, so that we find 
\begin{align*} 
  S(n,p^k)= A_{2^n-2}(\beta_n^k), \quad\quad  & U(n,p^k)= A_{2^n-1}(\beta_n^k), \nonumber\\
  V(n,p^k)= B_{2^n-2}(\beta_n^k),\quad\quad
  & W(n,p^k)= B_{2^n-1}(\beta_n^k).\label{eq:SUVWAB}
\end{align*}
First, notice that 
$B_{2^n-1}(\beta_{n}^k) = W(n,p^k) =\frac 1 {p^{k(2^n-1)}}$, hence $\tilde B_{2^n-1}(\beta_{n}^k)=1$.

 By \eqref{eq:Stilde},  $\tilde{S}(1,X)=1$ and $\tilde{S}(2,X)=-1$; moreover, using an induction argument, for every $n\ge 3$ we have that $\tilde S(n,X)$ is a polynomial of degree $3\cdot 2^{n-2}-2$ of the form
 $$ \tilde S(n,X)=-(1+2\sum X^{2i_j}),$$
 for some indexes $1\le i_j\le 3\cdot 2^{n-2}-2$. 
 Using this property we deduce $|\tilde S(1, p^k)|=|\tilde S(2, p^k)|=1$ and, for $n\ge 3$, 
 \begin{equation}\label{eq:maggiorazione1} 
  |\tilde S(n, p^k)| \leq 2\left |\sum_{j=0}^{3\cdot 2^{n-3}-1} p^{2jk}\right |\leq  2\left |\frac {p^{3\cdot 2^{n-2}k}-1} {p^{2k}-1}\right |< 2p^{3\cdot 2^{n-2}k}. 
 \end{equation}
 Furthermore, for $n \ge 1$, we have
 \begin{equation} \label{eq:maggiorazione2} 
 |\tilde U(n,p^k)| > p^{2^nk}. 
 \end{equation}
\medskip
\noindent To prove niceness it remains to verify condition $b$) in Definition \ref{def:nice1}. We have
\begin{equation*}
\frac {A_{2^{n}-1}(\beta_n^k)}{A_{2^{n}-2}(\beta_n^k)} = \frac {U(n,p^k)}{S(n,p^k)}=
\frac {\tilde U(n,p^k)}{p^k \tilde S(n,p^k)}.
\end{equation*}
For $n=1$ we have
$$\left | \frac {A_{1}(\beta_1^k)}{A_{
{0}}(\beta_1^k)}\right | =\left | \frac {\tilde U(1,p^k)}{p^k\tilde S(1,p^k)}\right |=\left |\frac {p^{2k}+1}{p^k}\right |> p^k > \frac 4 p, $$
and for $n=2$, 
$$\left | \frac {A_{3}(\beta_1^k)}{A_{
{2}}(\beta_1^k)}\right | =\left | \frac {\tilde U(2,p^k)}{p^k\tilde S(2,p^k)}\right |=\left |\frac {p^{4k}+p^{2k}+1}{p^k}\right |> p^{3k} 
> \frac 4 p.$$
Using \eqref{eq:maggiorazione1} and \eqref{eq:maggiorazione2}, for $n\geq 3 $ we finally have
\begin{equation*}
\left | \frac {A_{2^{n}-1}(\beta_n^k)}{A_{2^{n}-2}(\beta_n^k)}\right | =
\left |\frac {\tilde U(n,p^k)}{p^k \tilde S(n,p^k)}\right |> \left | \frac {p^{2^nk}}{2{p^{3\cdot 2^{n-2}k}} } \right | = \frac 1 2  p^{k(2^n-3\cdot 2^{n-2})}= \frac 1 2 p^{2^{n-2} k} >\frac 4 p,
\end{equation*}
proving the claim.
\end{proof}

\begin{remark}
It turns out that the sequence $\beta_n^k$ represents the $BCF$ expansion of the rational number
$$\frac {U(n,p^k)}{W(n,p^k)}= \frac {1+p^{2k}+p^{4k}+\ldots + p^{2^n k}}{p^k}.$$
We also point out that $\beta_n^k$ can be defined by a \lq\lq folding\rq\rq\  procedure, as described in  \cite[Theorem 1]{PooSh}.
\end{remark}

Finally, Theorem \ref{prop:tuno} and Proposition \ref{prop:betanice} implies the following result:
\begin{theorem}\label{teo:dueenne} 
For every $n,k \geq 1$ there are infinitely many integers $m$ coprime to $p$ such that the $BCF$ expansion for $p^k\sqrt{m}$ has period $2^n$. \end{theorem}

\subsection{Comparison with Ruban expansion of square roots}

We conclude this section by making a parallelism with the periodicity of numbers with trace $0$ in the case of Ruban. In \cite[Corollary 6.10]{CVZ2019}, the authors proved that a certain class of numbers of the form $\alpha=\sqrt{1+kp^h}$, with $(k,p)=1$ cannot have periodic $RCF$ expansion. Actually, up to now it is not know any example of a number of the form $\sqrt{\Delta}$, with $(\Delta, p)=1$ with periodic Ruban continued fraction expansion.
We continue this investigation in the following proposition. 

\begin{proposition} \label{prop:periodicity_Ruban} \
\begin{itemize}
    \item There exist infinitely many $\alpha$ with tr$(\alpha)=0$ and $v_p(\alpha)<0$ having periodic RCF;
    \item If $\alpha=p^k\sqrt{m}$ for some $m\in \ZZ \setminus p\ZZ$ and $k>0$, then the RCF expansion is not periodic.
\end{itemize}
\end{proposition}

\begin{proof}
For the first statement, we exhibit an example of a family of $\alpha$ of the first type such that the RCF expansion is periodic. Indeed, for $h\ge 0$ let us consider $\delta \in \QQ_p$ the square root of $1+p^h$ which is congruent to $1$ modulo $p$ and let us take $\alpha_h=\frac{\delta}{p^h}$; 
then, the Ruban continued fraction expansion of $\alpha_h$ is equal to $\left [ \frac{1}{p^h}, \overline{\frac{2}{p^h}} \right ]$. To prove this, notice that $a_0=\frac{1}{p^h}$, hence
$$ \alpha_1= \frac{1}{p^h}+\alpha; $$
from this we have $a_1=\frac{2}{p^h}$ and $\alpha_1=\alpha_2$, which proves the desired periodicity.
\\

To prove the second statement, consider $\alpha=p^k\sqrt{m}\in \QQ_p$ with $k>0$; then, $v_p(\alpha)>0$, hence $a_0=0$ and $\alpha_1=\frac{1}{\alpha}=\frac{\sqrt{m}}{p^k m}$ with $v_p(\alpha_1)<0.$ Now $a_1$ will be of the form $\frac{\tilde{a}_1}{p^k}$ for some $0< \tilde{a}_1 \le p^{k+1}$. Then we have that 
$$ \alpha_2= p^k \frac{\sqrt{m}+\tilde{a}_1m}{(1-\tilde{a}_1^2m)}, $$
where we recall that from Ruban algorithm we have that $p^k \mid (1-\tilde{a}_1m)$. But now if we consider the two real embeddings of $\alpha_2$ this are both negative, which applying \cite[Theorem 6.5]{CVZ2019} gives that $\alpha_2$ is not periodic as wanted.
\end{proof}

\section{Some further periodic expansions} \label{sec:further}

In this section, we give some other families of quadratic irrationals having periodic expansions of period $4$ and $6$, which are not obtained applying the construction given in Theorem \ref{prop:tuno}. 



\begin{theorem}
\ 
\begin{enumerate}
\item For $p \geq 3$ and $t \geq 2$, we have
\[ \cfrac{\sqrt{1-p^{t+2}}}{2p} = \pm\left[\cfrac{p^2-1}{2p}, \overline{-\cfrac{2}{p}, -\cfrac{1}{p^{t-1}}, \cfrac{2}{p}, -\cfrac{1}{p}}\right]; \]
\item For $p \geq 5$ and $t \geq 3$, we have
\[ \cfrac {\sqrt{p^t+1}} 2 =\pm\left [-\frac {p-1} 2, \frac 2 p,\overline{ -\frac 1 {p^{t-2}}, 1-\frac 2 p, -1-\frac 2 p, \frac 1 {p^{t-2}}, -1+\frac 2 p, 1+\frac 2 p}\right ]; \]
\item For $p \geq 5$ and $t \geq 3$, we have
\[ \cfrac {\sqrt{p^t+1}} {2p^{t-2}}=\pm \left [\cfrac {p^{t-1}+1}{2p^{t-2}},\overline{1-\frac 2 p, -1-\frac 2 p, \frac 1 {p^{t-2}}, -1+\frac 2 p, 1+\frac 2 p, -\frac 1 {p^{t-2}}}\right ]. \]
\end{enumerate}
\end{theorem}
\begin{proof}
The period of this continued fraction converges to an irrational $\alpha_1$ that can be evaluated considering the matrix
\[ M = \begin{pmatrix} M_{11} & M_{12} \cr M_{21} & M_{22} \end{pmatrix} = \begin{pmatrix} -\frac{2}{p} & 1 \cr 1 & 0 \end{pmatrix}\begin{pmatrix} -\frac{1}{p^{t-1}} & 1 \cr 1 & 0 \end{pmatrix}\begin{pmatrix} \frac{2}{p} & 1 \cr 1 & 0 \end{pmatrix}\begin{pmatrix} -\frac{1}{p} & 1 \cr 1 & 0 \end{pmatrix}. \]
Its characteristic polynomial is $x^2 + 2 \left(\frac{2}{p^{t+2}}-1\right)x+1$, and the eigenvalues are
\[ \mu_{1,2} = 1 - 2 p^{-t-2} \pm \sqrt{p^{-2t-4} - p^{-t-2}}. \]
Then
\[\alpha_1 = \lim_{n} \cfrac{A \mu_1^n + B \mu_2^n}{C\mu_1^n + D\mu_2^n} = \cfrac{A}{C},\]
where the limit is computed with respect to the $p$-adic norm and $\mu_1$ is the eigenvalue with larger $p$--adic norm. Moreover, the coefficients $A, B, C, D$ are the solutions of the systems
\[
\begin{cases} A + B = 1 \cr A\mu_1 + B\mu_2 = M_{11} \end{cases},\quad \begin{cases} C + D = 0 \cr C\mu_1 + D\mu_2 = M_{21} \end{cases},
\]
from which
\[ \alpha_1 = \cfrac{\mu_2 - M_{11}}{M_{21}}, \]
and the limit of the continued fraction is $\alpha_0 = \cfrac{p^2-1}{2p} + \cfrac{1}{\alpha_1}$. By direct calculation one can check that $\alpha_0 = \cfrac{1-p^{t+2}}{2p}$.
For cases 2. and 3. the proof is similar once the suitable matrix $M$ is considered.
\end{proof}

\end{document}